\documentclass[12pt]{amsart}
\usepackage{amssymb, a4wide, mathdots,url,hyperref, shuffle}
\usepackage{bbm}
\usepackage{tikz}
\usepackage{ytableau}
\usepackage{pb-diagram}
\newcommand{\Ad}{\operatorname{Ad}}

\newcommand{\Tr}{\operatorname{Tr}}

\newcommand{\Hom}{\operatorname{Hom}}
\renewcommand{\subset}{\subseteq}

\newcommand{\Wh}{\operatorname{Wh}}

\newtheorem{theorem}{Theorem}[section]
\newtheorem{lemma}[theorem]{Lemma}
\newtheorem{proposition}[theorem]{Proposition}
\newtheorem{remark}[theorem]{Remark}

\date{\today}

\renewcommand{\subset}{\subseteq}

\newcommand{\sspan}{\operatorname{span}}

\newcommand{\End}{\mathrm{End}}

\newcommand{\resu}{\mathrm{res}_u}

\newcommand{\ch}{\mathrm{ch}}

\begin{document}

\title[Triangular system for local character expansions]{A triangular system for local character expansions of Iwahori-spherical representations of general linear groups}


\begin{abstract}
For Iwahori-spherical representations of non-Archimedean general linear groups, Chan-Savin recently expressed the Whittaker functor as a restriction to an isotypic component of a finite Iwahori-Hecke algebra module. We generalize this method to describe principal degenerate Whittaker functors. Concurrently, we view Murnaghan's formula for the Harish-Chandra--Howe character as a Grothendieck group expansion of the same module.

Comparing the two approaches through the lens of Zelevinsky's PSH-algebras, we obtain an explicit unitriangular transition matrix between coefficients of the character expansion and the principal degenerate Whittaker dimensions.

\end{abstract}

\author{Maxim Gurevich}
\address{Department of Mathematics, Technion -- Israel Institute of Technology, Haifa, Israel.}
\email{maxg@technion.ac.il}

\maketitle

\section{Introduction}

Let $G_n = GL_n(F)$ be the general linear locally compact group, defined over a $p$-adic field $F$. Let $\mathfrak{g}_n = \mathfrak{gl}_n(F)$ be its Lie algebra. Integer partitions $P(n)$ naturally parameterize nilpotent $\Ad(G_n)$-orbits $\{\mathcal{O}_{\alpha}\}_{\alpha\in P(n)}$ in $\mathfrak{g}_n$.

We focus on two sets of integer invariants, indexed by those nilpotent orbits, attached to each smooth complex irreducible representation $\pi$ of $G_n$.

The first is the Harish-Chandra--Howe local character expansion \cite{howe74,hcbook}: The trace of $\pi$, viewed as a distribution on $G_n$, is known to be represented by a locally constant integrable function $\Theta_\pi$ on regular elements of the group. For a regular element $X\in \mathfrak{g}_n$ close enough to zero, a celebrated expansion
\begin{equation}\label{eq:expa}
\Theta_\pi(1+ X) = \sum_{\alpha\in P(n)} c_\alpha(\pi) \widehat{\mu}_{\mathcal{O}_\alpha}(X),\quad c_\alpha(\pi)\in \mathbb{Z}\;,
\end{equation}
is known to hold, where $\widehat{\mu}_{\mathcal{O}_\alpha},\, \alpha\in P(n)$, are suitably normalized functions representing the $\mathfrak{g}_n$-Fourier transform of the orbital integral distribution coming from the orbit $\mathcal{O}_\alpha$.

The second set of invariants are dimensions of certain degenerate Whittaker models attached to $\pi$ and a nilpotent $\Ad(G_n)$-orbit. Namely, for $\alpha=(\alpha_1,\ldots,\alpha_k) \in P(n)$, a standard Levi subgroup $G_{\alpha_1}\times \cdots \times G_{\alpha_k}\cong  M_{\alpha}  < G_n$ is attached, and an exact Jacquet functor $\mathbf{r}_{\alpha}$ is defined, which produces a finite-length smooth $M_\alpha$-representation $\mathbf{r}_\alpha(\pi)$.

We denote the integers
\[
d_\alpha(\pi) = \dim_{\mathbb{C}} \Wh(\mathbf{r}_{\alpha}(\pi)),\; \alpha\in P(n)\;,
\]
where $\Wh$ stands for the Whittaker functor on $M_\alpha$-representations. 

The spaces $\Wh(\mathbf{r}_{\alpha}(\pi))$, to which we will refer as principal degenerate Whittaker models, received attention as early as the foundational work of Zelevinsky \cite{Zel}. They were incorporated in the more general framework of degenerate Whittaker models of Moeglin--Waldspurger \cite{mw87}, and more recently were shown in \cite{ggs} to be minimal, in the proper sense, of such models.

Explicit formulas for values of the dimensions $d_\alpha(\pi)$ in special classes of representations were recently explored in \cite{mitra-sayag,mitra-deg}.

By exactness and uniqueness properties of the Whittaker functor, those invariants may be described as the number of (Whittaker-)generic irreducible constituents of $\mathbf{r}_{\alpha}(\pi)$. They may also be expressed in a slightly different form: For each $0\leq k\leq n$, the exact Bernstein--Zelevinsky derivative functor $\sigma \mapsto \sigma^{(k)}$ takes smooth $G_n$-representations to smooth $G_{n-k}$-representations. Here, we treat $G_0$-representations merely as vector spaces.

In these terms $d_{\alpha}(\pi)$ counts the dimension of the iterated derivative 
\[
(\cdots(\pi^{(\alpha_1)})^{(\alpha_2)}\cdots)^{(\alpha_k)}\;.
\]

This note explicates a linear formula for transition between the invariants $\{c_\alpha(\pi)\}_{\alpha\in P(n)}$ and $\{d_\alpha(\pi)\}_{\alpha\in P(n)}$ for the case when $\pi$ is a Iwahori-spherical representation. 


Let $I_n< G_n$ denote an Iwahori subgroup, which is contained in a maximal compact subgroup $I_n< K_n = GL_n(\mathcal{O}_F)< G_n$ ($\mathcal{O}_F$ stands for the ring of integers of $F$). An irreducible $G_n$-representation $\pi$ is called Iwahori-spherical if its subspace $\pi_0 = \pi^{I_n}$ of $I_n$-invariant vectors is non-zero.

Iwahori-spherical irreducible representations are known to be precisely those which appear in the principal Bernstein block of $G_n$-representations.

\begin{theorem}\label{thm:main}
Assume that the residual characteristic of the field $F$ is greater than $2n$.

For each Iwahori-spherical irreducible representation $\pi$ of $G_n$ and each partition $\alpha\in P(n)$, the formula
\[
d_\alpha(\pi) = \sum_{\beta\in P(n)} s(\alpha,\beta^t) c_\beta(\pi)
\]
holds, where $s(\alpha,\beta)\in \mathbb{Z}_{>0}$ are familiar combinatorial invariants described in \eqref{eq:graph}, and $\beta \mapsto \beta^t$ is the combinatorial transposition involution on $P(n)$.

\end{theorem}

The resulting transition matrix $(s(\alpha,\beta^t))_{\alpha,\beta}$ is evidently unitriangular, with respect to the natural partial order on the set of partitions $P(n)$. This is also the topological order on nilpotent orbits, i.e. $\mathcal{O}_\alpha \subset \overline{\mathcal{O}_\beta}$ is equivalent to $\alpha\leq \beta$.

Indeed, the combinatorial classification of Zelevinsky identifies, for an irreducible representation $\pi$, a unique maximal $\alpha_\pi\in P(n)$ with a non-zero $d_{\alpha_\pi}(\pi)$ (in fact, $d_{\alpha_\pi}(\pi)= 1$). The $GL_n$ case of the celebrated result of Moeglin--Waldspurger \cite{mw87} then showed that $\alpha= \alpha_\pi$ is also the maximal partition for which the coefficient $c_{\alpha}(\pi)$ does not vanish (see e.g. \cite{dimasahi-surv}), and that $c_{\alpha_\pi}(\pi)= 1$ . The orbit $\mathcal{O}_{\alpha_\pi}$ thus becomes the wave-front set of $\pi$ in the general formalism for representations of reductive $p$-adic groups.

Theorem \ref{thm:main} is therefore a quantitative refinement, for the Iwahori-invariant case, of this classical comparison. An abstract general triangular expression of the local character expansion was suggested by Barbasch-Moy \cite{bm97}, which was a main influence for this work. Similarly to our proof, their treatment involved a reduction to finite group analogues of degenerate Whittaker models.

Let us elaborate on the line of reasoning for the proof of Theorem \ref{thm:main}. 

The finite-dimensional semisimple Hecke algebra $H_n$ of $I_n$-bi-invariant complex functions on $K_n$ acts naturally on the space $\pi_0$.

The key argument is that both sets of our invariants of interest for $\pi$ are encoded in the $H_n$-module $\pi_0$. For local character expansion coefficients this phenomenon is interpreted out of Murnaghan's elegant formulas \cite{murn-crelle} for depth-zero representations (for large enough residual characteristic of $F$), which are based on expressions from \cite{walds-hall} of orbital integrals on parahoric subgroups.

As for dimensions of principal degenerate Whittaker models, their encoding in $\pi_0$ stems from the non-degenerate case in the works of Chan--Savin \cite{CS1} and Barbasch--Moy \cite{bm94}.

The comparison between the two lines of work mentioned above becomes transparent when put into the context of Zelevinsky's PSH-algebras \cite{zel-book} approach to the representation theory of $H_n$.

More precisely, let us write $\mathfrak{R}_n$ for the Grothendieck group of $H_n$-modules. The sum $\mathfrak{R} = \oplus_{n\geq0} \mathfrak{R}_n$ now becomes a (commutative) ring, relative to a natural induction structure. This ring is known as the universal PSH-algebra, or the Hall algebra when put into the correct context.

The PSH-algebra framework provides two bases $\mathcal{X} = \{x_\alpha\}_{\alpha \in P(n)}$ and $\mathcal{Y} = \{y_\alpha\}_{\alpha \in P(n)}$ for $\mathfrak{R}_n$ and a perfect pairing on the group, by which its structure is axiomatically analysed.

In this context we show that Murnaghan's formula amounts to $\{c_\alpha(\pi)\}_{\alpha\in P(n)}$ being the expansion of $[\pi_0]\in \mathfrak{R}_n$ in the basis $\mathcal{X}$, while the degenerate analogue of \cite{CS1} puts $\{d_\alpha(\pi)\}_{\alpha\in P(n)}$ as the expansion of $[\pi_0]$ in the basis dual to $\mathcal{Y}$, relative to the pairing. Thus, Theorem \ref{thm:main} follows from a computation of the triangular basis transition matrix, which is a combinatorial exercise handled in Section \ref{prop:graph}.

We hope that our result may be extended in future work into a more general understanding of the elusive links between Harish-Chandra--Howe character distributions, Iwahori--Hecke algebra representation theory and invariants arising from degenerate Whittaker models. In that context we mention a related work of Ciubotaru--Mason-Brown--Okada \cite{ciub21} on the role of Arthur packets in the description of the wavefront set for Iwahori-spherical representations.

\subsection{Acknowledgements}
Special thanks to Dan Ciubotaru for insightful remarks that sparked my interest in the problem, and to Fiona Murnaghan for an encouraging discussion, and for sharing her point of view. Thanks are also due to Dima Gourevitch, Kei Yuen Chan, Jiandi Zou and Chuijia Wang for valuable discussions.

\section{Background}

Let us write
\[
C(n) = \left\{ (\alpha_1, \ldots,\alpha_k)\,:\, \alpha_i\in \mathbb{Z}_{>0},\; \alpha_1+\ldots + \alpha_k = n\right\}
\]
for the set of compositions of an integer $n\geq1$, and
\[
P(n) = \{(\alpha_1\geq \ldots \geq \alpha_k)\} \subset C(n)
\]
for the set of its partitions.

For $\alpha=(\alpha_1,\ldots,\alpha_k)\in C(n)$, the standard Levi subgroup $M_{\alpha} <G_n$ is the group of block-diagonal matrices, with block sizes $\alpha_1,\ldots,\alpha_k$. We write $P_\alpha < G_n$ for the standard parabolic subgroup generated by $M_\alpha$ and the upper-triangular matrices in $G_n$, and $N_\alpha< P_\alpha$ for its unipotent radical.

Denoting the ring of integers of $F$ by $\mathcal{O}_F$ and its maximal ideal by $\mathfrak{p}_F$, we consider the maximal compact subgroup $K_n = GL_n(\mathcal{O}_F)< G_n$ and its normal subgroup $K^1_n = I + M_n(\mathfrak{p}_F) < K_n$. Clearly, $K_n/K_n^1 \cong \overline{G}_n:=GL_n(\mathbb{F})$, where $\mathbb{F}$ is the finite residue field $\mathbb{F} = \mathcal{O}_F/\mathfrak{p}_F$.

\subsection{Finite general linear groups}

Let us recall some aspects of the complex representation theory of the finite group $\overline{G}_n$. We follow the elegant treatment of Zelevinsky's book \cite{zel-book} and review some of its results.

We write $\overline{P}_\alpha = \overline{M}_\alpha \overline{N}_\alpha< \overline{G}_n$, for each $\alpha\in C(n)$, for the analogous finite field versions of standard parabolic subgroups and their Levi decompositions. We also write $B_n = P_{(1,\ldots,1)}$ for the minimal standard parabolic subgroup.

Given a composition $\alpha= (\alpha_1,\ldots,\alpha_k)\in C(n)$ and a tuple of representations $\sigma_i$ of $G_{\alpha_i}$, for $i=1,\ldots,k$, the parabolic induction $\sigma_1\times \cdots\times \sigma_k$ is defined as the $G_n$-representation induced from the inflation of $\sigma_1\otimes\cdots \otimes \sigma_k$ to $P_\alpha$.

An irreducible representation of $\overline{G}_n$ is called \textit{unipotent}, if it possesses non-zero $B_n$-invariant vectors.

Let $\mathfrak{R}_n$ denote the Grothendieck group of finite-dimensional complex representations of $\overline{G}_n$, whose irreducible constituents are all unipotent. We write $[\sigma]\in \mathfrak{R}_n$ for the isomorphism class of a $\overline{G}_n$-representation $\sigma$.

A non-degenerate symmetric bilinear form on $\mathfrak{R}_n$ is given by
\[
\langle [\sigma_1],[\sigma_2]\rangle = \dim \Hom_{H_n} (\sigma_1,\sigma_2)\;.
\]

Zelevinsky identifies the sum of abelian groups
\[
\mathfrak{R} = \oplus_{n\geq0} \mathfrak{R}_n
\]
with an axiomatic notion of a \textit{universal positive self-adjoint Hopf algebra}.

In particular, it becomes a commutative associative ring with respect to the parabolic induction product $[\sigma_1][\sigma_2]:= [\sigma_1\times \sigma_2]$. Here, $\mathfrak{R}_0 = \mathbb{Z}$ is viewed formally as the ring identity element.

We write $x_n\in \mathfrak{R}_n$ for the class that corresponds to the trivial representation of $\overline{G}_n$. The irreducible Steinberg representation of $\overline{G}_n$, whose class we write as $y_n\in \mathfrak{R}_n$, plays a role dual to the trivial representation.

Each partition $\alpha\in (\alpha_1,\ldots,\alpha_k)\in P(n)$ gives rise to product elements
\[
x_\alpha = x_{\alpha_1}\cdots x_{\alpha_k},\;y_\alpha = y_{\alpha_1}\cdots y_{\alpha_k}\in \mathfrak{R}_n\;.
\]
For each $n\geq1$, both sets of elements
\[
\mathcal{X} = \{x_\alpha\}_{\alpha\in P(n)},\quad \mathcal{Y} = \{y_\alpha\}_{\alpha\in P(n)}
\]
give bases to the free abelian group $\mathfrak{R}_n$.

Let $\alpha\in (\alpha_1,\ldots,\alpha_k), \;\beta = (\beta_1,\ldots,\beta_l)\in P(n)$ be two given partitions of an integer $n\geq1$. We set the invariant
\begin{equation}\label{eq:graph}
s(\alpha,\beta) = \# \left\{ A \subset \{1,\ldots, k\} \times \{1,\ldots, l\} \;:\; \begin{array}{cc}
\alpha_ i = \# \{ j\;:\; (i,j)\in A\},\; \forall 1\leq i \leq k  \\
\beta_ j = \#\{ i\;:\; (i,j)\in A\},\; \forall 1\leq j \leq l  \end{array} \right\}\;.
\end{equation}
In other words, $s(\alpha, \beta)$ counts the number of bipartite graphs with labelled vertices, whose vertex degrees are prescribed by the given partitions.

\begin{proposition}\cite[3.17(c)]{zel-book}\label{prop:graph}
The pairing in $\mathfrak{R}_n$ satisfies
\[
\langle x_\alpha, y_\beta \rangle = s(\alpha,\beta)\;,
\]
for all $\alpha,\beta\in P(n)$.

\end{proposition}

For $\alpha= (\alpha_1,\ldots,\alpha_k)\in P(n)$, we write the transposed partition $\alpha^t = (\beta_1,\ldots,\beta_l)$ to be given by $\beta_i  = \#\{ 1\leq j\leq k\;:\; \alpha_j \geq i\}$. Evidently, $(\alpha^t)^t = \alpha$.

Recall that the set of partitions $P(n)$ is equipped with a partial order defined by the dominance relation. 

A moment's reflection shows that $s(\alpha, \alpha^t)=1$ for all $\alpha \in P(n)$, and that $s(\alpha,\beta^t) = 0 $ for $\alpha  > \beta\in P(n)$. Hence, the transition matrix between the bases $\mathcal{X}$ and $\mathcal{Y}$, when ordered by $P(n)$-labels, becomes unitriangular with respect to that order.

\subsubsection{Class functions}
Let $\mathcal{C}(\overline{G}_n)$ be the (finite-dimensional) space of conjugation-invariant complex functions on the finite group $\overline{G}_n$ that are supported on unipotent elements.

Given a finite-dimensional complex representation $\sigma$ of $\overline{G}_n$, we set $\ch_u(\sigma)\in \mathcal{C}(\overline{G}_n)$ to be the restriction of the character function $g\mapsto \Tr(\sigma(g))$ to the unipotent elements of $\overline{G}_n$.

We obtain a linear map
\begin{equation}\label{eq:iso}
\resu: \mathbb{C}\otimes \mathfrak{R}_n \;\to \; \mathcal{C}(\overline{G}_n)\;,
\end{equation}
which takes an isomorphism class $[\sigma]\in \mathfrak{R}_n$ to the restricted function $\ch_u(\sigma)$.

\begin{lemma}\cite[10.3]{zel-book}\label{lem:isom}
The map $\resu$ is a linear isomorphism.
\end{lemma}

\subsection{Hecke algebra perspective}

Let $\tau_n$ denote the unipotent representation of $\overline{G}_n$, for which $[\tau_n] = x_{(1,\ldots,1)}$ holds in $\mathfrak{R}_n$. In other words, $\tau_n$ is the induction of the trivial representation of $B_n$.

It follows that the complex finite-dimensional intertwiner algebra
\[
H_n = \End_{\overline{G}_n}(\tau_n)
\]
acts on the space of $B_n$-invariants for each unipotent $\overline{G}_n$-representation.

The resulting functor is known to give a bijection on irreducible representation (for example, \cite[Theorem 6.1.1]{michel-book}). More precisely, taking a $\overline{G}_n$-representation to its $B_n$-invariants identifies $\mathfrak{R}_n$ with the Grothendieck group of (all) complex finite-dimensional representations of the algebra $H_n$. In particular, we will also write $[\sigma]\in \mathfrak{R}_n$, for a representation $\sigma$ of $H_n$.

We recall that $H_n$ is the Iwahori-Hecke algebra (of type $GL_n$), and can also be viewed as a deformation of the group algebra of the symmetric group $\mathfrak{S}_n$. Furthermore, it is known to be isomorphic to it (see \cite[Chapter 6.2]{michel-book}).

We note that the Steinberg representation of $\overline{G}_n$ has a $1$-dimensional space of $B_n$-invariants, which produces the $H_n$-representation corresponding to the sign representation of $\mathfrak{S}_n$, under the above isomorphism. Thus, viewing $\mathfrak{R}_n$ as the Grothendieck group of $H_n$, the elements $x_n,y_n\in \mathfrak{R}_n$ stand for the isomorphism classes of the two unique $1$-dimensional representations.

\subsubsection{Product structure}\label{sect:adjun}
For any $\alpha=(\alpha_1,\ldots,\alpha_k)\in C(n)$, let us denote the algebra
\[
H_\alpha  := H_{\alpha_1}\otimes \cdots \otimes H_{\alpha_k}\;.
\]
Since $\tau_n = \tau_{\alpha_1} \times \cdots\times \tau_{\alpha_k}$, we have an embedding of algebras
\begin{equation}\label{eq:embed}
\iota_{\alpha}: H_\alpha \cong \End_{\overline{M}_{\alpha}}(\tau_{\alpha_1} \otimes \cdots\otimes \tau_{\alpha_k}) \;\to H_n\;.
\end{equation}

For a tuple of representations $\sigma_i$ of $H_{\alpha_i}$, $i=1,\ldots,k$, the induction product $ \sigma_1 \times \cdots \times \sigma_k$ is defined as the $H_n$-representation induced from $\sigma_1\otimes \cdots \otimes \sigma_k$ through the above embedding.

Since $B_n = (\overline{M}_{\alpha}\cap B_n) \overline{N}_{\alpha}$ holds, in the ring $\mathfrak{R}$, we clearly have 
\[
[\sigma_1\times\cdots\times \sigma_k] = [\sigma_1]\cdot\cdots \cdot [\sigma_k],
\]
with the latter product already defined through $\overline{G}_n$-representations.

We also let $\widehat{\mathbf{r}}_\alpha$ be the functor taking $H_n$-representations to $H_\alpha$-representations by restricting through the embedding $\iota_\alpha$.

Due to semisimplicity of all algebras involved, the functor $\widehat{\mathbf{r}}_\alpha$ is both right and left adjoint to the induction product functor.

\section{Murnaghan character expansion}

For $\alpha\in P(n)$, let $\mathfrak{n}_{\alpha} < \mathfrak{g}_n$ be the Lie algebra (consisting of block-upper-triangular matrices) of the unipotent radical $N_\alpha$ of $P_\alpha$.

We set $\mathcal{O}_\alpha$ to be unique nilpotent $\Ad(G_n)$-orbit in $\mathfrak{g}_n$, for which $\mathfrak{n}_{\alpha^t} \cap \mathcal{O}_{\alpha}$ is dense in $\mathfrak{n}_{\alpha^t}$.

\begin{remark}\label{rem:choice}
The duality involved in our notation is natural from the geometric point of view, so that the dominance order on partition corresponds to the topological closure order on nilpotent orbits. 

In terms of matrices, one can take $\mathcal{O}_\alpha$ to be the nilpotent orbit of the Jordan matrix composed of block sizes that are desribed by $\alpha$. The nilpotent matrix of the Weyr canonical form that is described by $\alpha$ then belongs to $\mathfrak{n}_{\alpha} \cap \mathcal{O}_{\alpha^t}$.

We note that this choice differs from the notations of \cite{murn-crelle}, on which we base our other conventions.

\end{remark}

Integrating over the $G_n$-invariant measure on $\mathcal{O}_{\alpha^t}$ defines a distribution $\mu_{\mathcal{O}_\alpha}\in C_c^\infty(\mathfrak{g}_n)^\ast$, whose Fourier transform is a distribution described by the function $\widehat{\mu}_{\mathcal{O}_\alpha}$ on regular elements of $\mathfrak{g}_n$.

We normalize all measures involved to fit the conventions in \cite{murn-crelle}. In particular, the normalization is pinned by choosing $\widehat{\mu}_{\mathcal{O}_{\alpha^t}}$ to equal the character expansion near zero of the $P_\alpha$-parabolic induction of the trivial representation of $M_\alpha$.

For a smooth irreducible representation $\pi$ of $G_n$, the coefficients $c_\alpha(\pi)$ are now defined as in the identity \eqref{eq:expa} given in the introduction section.

\subsection{Restriction to the finite Hecke algebra}

Let us recall the relation of the Iwahori--Hecke algebras $H_n$ and the ring $\mathfrak{R}$ with the representation theory of the $p$-adic group $G_n$.

For any totally disconnected locally compact group $T_1$ and a open compact subgroup $T_2< T_1$, let us write $\mathcal{H}(T_1,T_2)$ for the space of compactly supported complex functions on $T_1$ that are bi-invariant with respect to $T_2$. It is naturally equipped with an associative (unital) convolution product.

We have a standard identification $H_n\cong \mathcal{H}(\overline{G}_n,B_n)$ for algebras of intertwiner operators.

The (standard) Iwahori subgroup $K^1_n < I_n < K_n$ is taken as the preimage of $B_n$ under the quotient map $K_n\to \overline{G}_n$. We thus obtain an identification
\begin{equation}\label{eq:ident}
H_n\cong \mathcal{H}(K_n,I_n)\;.
\end{equation}

Recall that a smooth irreducible representation of $G_n$ is called Iwahori-spherical, if it possesses non-zero $I_n$-invariant vectors.

For an irreducible Iwahori-spherical representation $(\pi,V)$ of $G_n$, the space of invariant vectors $V^{I_n}$ naturally becomes a $H_n$-representation through \eqref{eq:ident}, which we denote by $\pi_0$.

We now interpret a basic case of the main result of \cite{murn-crelle}.

\begin{proposition}\label{prop:murngh}

Assume that the residual characteristic of the field $F$ is greater than $2n$.

Let $\pi$ be an irreducible Iwahori-spherical representation of $G_n$, and $\pi_0$ be the resulting representation of the Iwahori-Hecke algebra $H_n$.

As elements of $\mathfrak{R}_n$, the identity
\[
[\pi_0] = \sum_{\alpha\in P(n)} c_\alpha(\pi) x_{\alpha^t}
\]
holds.

\end{proposition}

\begin{proof}

Let $V$ denote the space of $\pi$. We have $\{0\}\neq V^{I_n} \subseteq V^{K^1_n}$. A finite-group representation $\overline{\pi}$ of $\overline{G}_n$ on $V^{K^1_n}$ is obtained by factoring through the $\pi(K_n)$-action.

By (the depth zero case of) \cite[Lemma 4.5]{murn-crelle}, we now have
\[
\resu(\overline{\pi}) = \sum_{\alpha\in P(n)} c_\alpha(\pi) \resu(x_{\alpha^t})\;.
\]
Note, that viewing $V^{I_n}$ as the space of $B_n$-invariant vectors for $\overline{\pi}$, the representation $\pi_0$ can be directly constructed out of $\overline{\pi}$.

Moreover, since $\pi$ is Iwahori-spherical, we know by the general theory of types for $G_n$ (see \cite{mp94} or \cite{morris}) that $\overline{\pi}$ must be a unipotent representation. In other words, $[\overline{\pi}]\in \mathfrak{R}_n$ is a well-defined element (that is equal to $[\pi_0]$).

Hence, the formula now follows from the isomorphism in  Lemma \ref{lem:isom}.

\end{proof}

\section{Degenerate Whittaker dimensions}

Let $\psi: F  \to\mathbb{C}^\times$ be a fixed non-zero additive character. We write $U_n= N_{(1,\ldots,1)}  <G_n$ for the standard maximal unipotent subgroup.

Given a partition $\alpha \in P(n)$, let $I_{\alpha} =\{1,\ldots,n\} \setminus \{\alpha_1 + \ldots +\alpha_j\}_{j=1}^k$ be the corresponding set of indices. Then, $\psi_\alpha: U_n\to\mathbb{C}^\times$ is the group character taking a matrix $(u_{i,j})\in U_n$ to $\psi( \sum_{i\in I_\alpha} u_{i,i+1})$.

Each $\alpha\in P(n)$ thus defines the \textit{principal degenerate Whittaker functor} $\Wh_{\alpha}$ by taking a smooth representation $(\pi,V)$ of $G_n$ to the co-invariant vector space
\[
\Wh_{\alpha}(\pi) = V / \sspan \{ \pi(u)v-\psi_{\alpha}(u)v \;:\;v\in V,\, u\in U_n\}\;.
\]

In the general formalism of \cite{mw87,ggs}, the functor $\Wh_\alpha$ becomes associated with a degenerate model arising from the nilpotent orbit $\mathcal{O}_{\alpha}$. It again justifies our enumeration of nilpotent orbits as noted in Remark \ref{rem:choice}.

As defined in the introduction section, we record the dimensions of principal degenerate Whittaker models as
\[
d_\alpha(\pi) = \Wh_{\alpha}(\pi) \;.
\]

Similarly, the Jacquet functor produces a $M_\alpha$-representation on the co-invariant space
\[
\mathbf{r}_{\alpha}(\pi) = V / \sspan \{ \pi(u)v- v \;:\;v\in V,\, u\in N_\alpha \}\;.
\]

For $\alpha = (n)$, $I_\alpha = I$ and $\psi_{\alpha}$ is a non-degenerate character. In this case, $\Wh = \Wh_{\alpha}$ is the (non-degenerate) Whittaker functor.

The Whitakker functor on smooth $M_\alpha$-representations, which we denote as $\Wh$ as well, may be naturally identified with $\Wh \otimes\cdots \otimes \Wh$, when $M_\alpha$ viewed as a product of general linear groups.

It is easy to verify the natural functor identification
\begin{equation}\label{eq:jacq}
\Wh_\alpha = \Wh \circ \mathbf{r}_{\alpha}\;,
\end{equation}
for each $\alpha\in P(n)$.

\subsection{Hecke algebra interpretation}

Let us recall a key corollary of the results of Chan-Savin \cite{CS1} and Barbasch-Moy \cite{bm94}, when applied to the $GL_n$ case.

\begin{proposition}\cite[Corollary 4.5]{CS1}\label{prop:chansav}
Let $\pi$ be a smooth representation of $G_n$ of finite-length, whose irreducible subquotients are all Iwahori-spherical. Let $\pi_0$ be the resulting finite-dimensional representation of $H_n$ on the space of $I_n$-invariant vectors in $\pi$.

The dimension of the Whittaker model $\Wh(\pi)$ is given by the pairing value $\langle [\pi_0], y_n\rangle$ in $\mathfrak{R}_n$, against the element $y_n\in \mathfrak{R}_n$ representing the sign/Steinberg representation.
\end{proposition}

We would like to state an enhancement to the above proposition by showing that all dimensions $d_\alpha(\pi)$, $\alpha\in P(n)$, for Iwahori-spherical representations $\pi$, can be extracted out of the isomorphism class of $\pi_0$.

For that goal we need to produce the Hecke algebra analog of the Jacquet functor that appears in the description \eqref{eq:jacq}.

For a composition $\alpha=(\alpha_1,\ldots,\alpha_k)\in C(n)$, let $\tau_\alpha$ be the $\overline{M}_{\alpha}$-representation induced from the trivial representation of $B_n \cap \overline{M}_\alpha$. Thus, $\tau_\alpha \cong  \tau_{\alpha_1}\otimes\cdots \otimes \tau_{\alpha_k}$, when identifying $\overline{M}_{\alpha}$ with a group product $\times_{i=1}^k \overline{G}_{\alpha_i}$. We see a chain of natural algebra identifications
\begin{equation}\label{eq:ident-tens}
\mathcal{H}(K_n \cap M_\alpha, I_n\cap M_\alpha) \cong \mathcal{H}(\overline{M}_\alpha, B_n \cap \overline{M}_\alpha)\cong \End_{\overline{M}_\alpha}(\tau_\alpha) \cong H_\alpha;.
\end{equation}

For a finite-length smooth representation $\rho$ of $M_\alpha$, we set $\rho_0$ to be the $H_\alpha$-representation on the $ I_n\cap M_\alpha$-invariant vectors in $\rho$, obtained through the above identification with a convolution algebra.


\begin{lemma}\label{lem:jacqu}

For an irreducible Iwahori-spherical representation $\pi$ of $G_n$ and a composition $\alpha\in C(n)$, we have an isomorphism
\[
\widehat{\mathbf{r}}_\alpha(\pi_0)\cong (\mathbf{r}_\alpha (\pi))_0
\]
of $H_\alpha$-representations.

\end{lemma}
\begin{proof}
Let $V$ be the space of $\pi$, and $V_\alpha$ the space of its Jacquet module $\mathbf{r}_\alpha(\pi)$.

It is known (e.g. \cite[Proposition I.4.3]{mw-zel}) that the projection
\begin{equation}\label{eq:proj}
p_\alpha: V^{I_n}\to V_\alpha^{I_n\cap M_\alpha}
\end{equation}
is an isomorphism of vectors spaces.

Moreover, unpacking the claim in \cite[Proposition 5 and Remark 5]{bush-lms} for the case of $I_n$, we see that $p_\alpha$ intertwines the resulting actions of $\mathcal{H}(K_n \cap M_\alpha, I_n\cap M_\alpha)$. In more detail,
\[
p_\alpha(\pi(\iota_\alpha(f)) v ) = \pi(f) p_\alpha(v)
\]
holds, for all $v\in V^{I_n}$ and $f\in\mathcal{H}(K_n \cap M_\alpha, I_n\cap M_\alpha)$. Here, $\pi$ stands for the actions of the corresponding convolution algebras on invariant vectors.

\end{proof}

\begin{remark}
The normalization twist issues in \cite{bk98,bush-lms} and the commonly encountered choice of an opposite parabolic in isomorphisms of the form \eqref{eq:proj} are absent in our discussion, because of our focus on actions rising from representations restricted to the compact group $K_n$.
\end{remark}


We are ready to state the degenerate analog to Proposition \ref{prop:chansav}.

\begin{proposition}\label{prop:deg-hecke}
Let $\pi$ be an irreducible Iwahori-spherical representation of $G_n$, and $\pi_0$ be the resulting representation of the Iwahori-Hecke algebra $H_n$.

For any partition $\alpha\in P(n)$, the dimension of the corresponding degenerate Whittaker models space is given as
\[
d_{\alpha}(\pi) = \langle [\pi_0], y_\alpha \rangle\;.
\]

\end{proposition}
\begin{proof}

By the adjunction described in Section \ref{sect:adjun}, we may write
\begin{equation}\label{eq:proof}
\langle [\pi_0], y_\alpha \rangle = \dim \Hom_{H_\alpha} (\widehat{\mathbf{r}}_\alpha(\pi_0), \epsilon_{\alpha_1}\otimes \cdots\otimes \epsilon_{\alpha_k})\;,
\end{equation}
where $\alpha = (\alpha_1,\ldots,\alpha_k)$ and $\epsilon_{\alpha_i}$ stands for the sign representation of $H_{\alpha_i}$ (that is, $y_\alpha = [\epsilon_{\alpha_1}\times \cdots \times \epsilon_{\alpha_k}]$).

Let us write $\{\sigma^i_1 \otimes \cdots \otimes \sigma^i_k\}_{i=1}^t$ for the Jordan-H\"{o}lder series of $\mathbf{r}_\alpha(\pi)$, when $M_\alpha$ is identified with $\times_{i=1}^k \overline{G}_{\alpha_i}$.

By exactness of the functor of taking $I_n\cap M_\alpha$-invariants, we have 
\[
(\mathbf{r}_\alpha(\pi))_0 = \bigoplus_{i=1}^t (\sigma^i_1)_0 \otimes \cdots \otimes (\sigma^i_k)_0\;.
\]
Therefore, by \eqref{eq:proof} and Lemma \ref{lem:jacqu}, we may write
\[
\langle [\pi_0], y_\alpha \rangle = \sum_{i=1}^t \prod_{j=1}^k \dim_{G_{\alpha_j}} ((\sigma^i_j)_0, \epsilon_{\alpha_j})= \sum_{i=1}^t \prod_{j=1}^k  \langle [(\sigma^i_j)_0], y_{\alpha_j} \rangle \;.
\]
When applying Proposition \ref{prop:chansav}, the last value becomes equal to $d_{\alpha}(\pi)$. 

\end{proof}

Theorem \ref{thm:main} is now a direct consequence of the two expansions in Proposition \ref{prop:murngh} and Proposition \ref{prop:deg-hecke}, and the change of basis formula in Proposition \ref{prop:graph}.

\bibliographystyle{alpha}
\bibliography{propo2}{}

\end{document}